\documentclass[a4paper,11pt]{article}
\usepackage[utf8]{inputenc}
\usepackage{geometry}
\usepackage[]{hyperref}
\usepackage{titlesec}
\usepackage{etoolbox}
\usepackage{chngcntr}
\usepackage{graphicx}
\usepackage{lipsum}
\usepackage{xcolor}
\usepackage{amsmath,amssymb,amsfonts}
\usepackage{fancyhdr}
\usepackage{todonotes}
\usepackage[ruled,vlined,linesnumbered]{algorithm2e}

\usepackage{tikz}

\counterwithout{figure}{section}

\patchcmd{\thebibliography}{\chapter*}{\section*}{}{}
\titleformat{\chapter}[frame]{\normalfont}{}{14pt}{\LARGE\bfseries\filcenter}
\usepackage{amsthm}

\newtheorem*{thmA}{Theorem A}

\theoremstyle{plain}

\newtheorem{theorem}{Theorem}[section]
\newtheorem{lemma}[theorem]{Lemma}
\newtheorem{corollary}[theorem]{Corollary}
\newtheorem{prop}[theorem]{Proposition}

\newtheorem{example}{Example}

\newtheorem{definition}{Definition}
\newtheorem{remark}{Remark}

\newcommand{\R}{\mathbb{R}}

\title{Application of polynomial algebras to non-linear equation solvers}
\author{J. Canela 
\\ Insitut Universitari de Matemàtiques i Aplicacions de Castelló\\ and Departament de Matemàtiques, \\
        Universitat Jaume I, \and D. Pérez-Palau\thanks{ Both authors are supported by the project CIGE/2023/102 funded by Generalitat Valenciana. The first author is also supported by the project UJI-B2022-46 (Universitat Jaume I) and the project PID2023-147252NB-I00 financed by MICIU/AEI MCIN/AEI/10.13039/501100011033 and
by FEDER, UE.}\\Departament de Matemàtiques,\\ Univesitat Politècnica de Catalunya - Barcelona Tech,\\ C/ Colom, 11, 08222 Terrassa, Spain. \\}

\begin{document}

\maketitle

\begin{abstract}
This paper presents a novel application of Jet Transport, a high-order automatic differentiation technique, to enhance classical numerical methods, with a focus on Newton's method. We prove a central theorem establishing that, under appropriate conditions, applying Jet Transport within a Newton iteration doubles the number of correct coefficients in the Taylor series approximation of the solution. This theoretical result is then extended to the practical case where the exact solution is unknown, demonstrating the expected quadratic convergence (error reduction from \( \varepsilon \) to \( \varepsilon^2 \)) while simultaneously doubling the order of accuracy in the series expansion. The efficacy of the resulting Jet-Newton method is demonstrated through three illustrative examples: an academic problem validating the theoretical convergence rates, the solution of Kepler's equation, and a new continuation algorithm for computing zero-velocity curves in the circular restricted three-body problem. These examples showcase the method's capability to provide high-order semi-analytical approximations.
\end{abstract}

\section{Introduction}\label{Sec:intro}

During the 1990s, a new numerical technique emerged based on jet propagation through polynomial algebras. The core idea of this approach is to approximate a function $f(x)$ with an $N$-th order Taylor polynomial $P(\xi)$ such that $P(\xi) \approx f(x_0 + \xi)$. This polynomial approximation, called the \emph{jet} of $f$, enables a systematic framework for computing functional approximations.

Due to the connection between polynomial approximations and Taylor series, jets can compute derivatives of arbitrary functions. This dual capability has led to various names for the method, including: Jet Transport \cite{Alessi09,Berz99,chen20b}, Differential Algebra \cite{armellin18}, Automatic Differentiation \cite{rall_perspectives_2006}.

The implementation of this algebraic framework is critical for practical applications. Classically, a jet is represented as a truncated formal series:
$$
P(\xi) = \sum_{k=0}^N a_k \xi^k.
$$
Where $\xi\in \mathbb R$ and $a_k\in\mathbb R$.

Two fundamental implementation challenges arise: storage of jets, typically implemented via lexicographical ordering or tree-based data structures, and product operations.
Several implementations exist in the literature, with notable contributions from \cite{GODOT, Mamotreto16,Jorba05,Makino2006,rasotto_differential_2016}, each adopting different strategies for these computational primitives.

For jets in one variable, consider:
\[
A(\xi) = \sum_{i=0}^n a_i\xi^i \quad \text{and} \quad B(\xi) = \sum_{i=0}^n b_i\xi^i, ~~ b_0\neq 0.
\]
Let $P(\xi) = A(\xi)B(\xi) = \sum_{i=0}^n p_i \xi^i$ be their product and $Q(\xi) = A(\xi)/B(\xi) = \sum_{i=0}^n q_i \xi^i$ their quotient. The coefficients of $P(\xi)$ and $Q(\xi)$ are given by \cite{Jorba05}:
\begin{equation}\label{eq:polynomialBasics}
p_i = \sum_{j=0}^i a_{i-j}b_j, \quad q_i = \dfrac{a_i - \sum_{j=1}^i q_{i-j}b_j}{b_0}.
\end{equation}

Therefore, jets on one variable without term of order 0 form a field. For every non-zero jet $J$, there exists a multiplicative inverse. By leveraging the composition law, one can derive explicit formulae for an extended set of operations, including: trigonometric functions, exponentials, and logarithms, amongst others.

From Equation~\ref{eq:polynomialBasics}, it follows immediately that if a jet $A(\xi)$ contains no terms of order smaller than $m$, then neither does the quotient $A(\xi)/B(\xi)$. This is formalized in the following trivial lemma. Recall that given two functions $f$ and $g$ defined on a neighbourhood of $x_0\in \R$, we say that $f(x)=\mathcal{O}(g(x))$ as $x\rightarrow x_0$ if there exists an $\varepsilon> 0$ and an $M>0$ such that if $|x-x_0|<\varepsilon$ then $|f(x)|\leq M|g(x)|$.

\begin{lemma}\label{lem:ordrediv}
Let $f(x)$ and $g(x)$ be two continuous functions defined around $x=0$ such that $f(x) = \mathcal{O}(x^m)$ and $g(x) = \mathcal{O}(1)$ as $x\rightarrow 0$, with $g(0) \neq 0$. Then
\[
\frac{f(x)}{g(x)} = \mathcal{O}(x^m).
\]
\end{lemma}

\begin{remark}
This result is immediate from the definition of $\mathcal{O}$ notation and the Taylor expansion of $1/g(x)$. We include it to highlight the preservation of order in jet divisions.
\end{remark}

\begin{remark}
Along the paper we use the variable $x$ for functions and $\xi$ for the polynomial approximations. Hence, Lemma~\ref{lem:ordrediv} will mostly be used using the variable $\xi$.
\end{remark}

With these properties, the algebra of jets can be used to approximate functions numerically, analogous to how floating-point arithmetic approximates real numbers. While floating-point representation truncates values at a fixed number of significant digits, the jet approach truncates formal series at a specified order $N$, providing a polynomial approximation of the function.

Several authors have employed this approximation scheme in applied problems. One particularly noteworthy application is the propagation of solutions to ordinary differential equations (ODEs). When polynomial algebra is applied to differential equation propagators (such as Runge-Kutta or Taylor methods) \cite{Gimeno22}, the resulting variational equations are obtained as though the original numerical method had been applied directly. 

This technique has found applications across diverse fields. Berz \cite{Berz99} pioneered its use in studying particle accelerator dynamics. In astrodynamics, researchers have extensively employed it to analyze uncertainty propagation in spacecraft and asteroid trajectory predictions \cite{Alessi09, armellin18, valli14}. The methodology was further extended to filtering problems by Chen et al. \cite{chen20,chen20b}, who incorporated jet propagation into the classical Kalman filter framework. Recently, Fernández-Mora et al \cite{fernandez-mora_flow_2024} and  Barcelona et al \cite{barcelona_semianalytical_2024} used such techniques to compute, using the parametrization method, invariant manifolds of partially hyperbolic tori and heteroclinic connections of periodic orbits.

From a theoretical perspective, the foundations of jet transport for differential equations were rigorously established by Gimeno et al. Their work \cite{Gimeno17} demonstrated the computation of periodic orbits in delayed differential equations via jet transport. Subsequent research \cite{Gimeno22} achieved high-order expansions of invariant manifolds around high-dimensional tori, showcasing the method's capability for complex dynamical systems.

As evidenced by the literature, jet transport has been successfully applied to the propagation of differential equations and continuous dynamical systems. In these methods, a jet $P(\xi)$ of order $N$ is typically computed to satisfy $f(P(\xi)) = 0$ for a given function $f$ (e.g., an invariance equation or flow map equation under certain regularity conditions). Classical approaches, such as the order matching algorithm, achieve this by computing the coefficients of $P(\xi)$ sequentially through order-by-order matching.  

In this work, we propose a novel alternative: a Newton-type iteration that computes all coefficients simultaneously, offering potential advantages in convergence rate and numerical stability.

Consider a nonlinear equation $f(x, c) = 0$ depending on a parameter $c$. A numerical solution can be obtained via Newton's method:

$$
x_{i+1} = N(x_i) = x_i - \frac{f(x_i, c)}{f_x(x_i, c)},
$$
where $f_x$ denotes the partial derivative with respect to $x$. Given an initial guess $x_0$ sufficiently close to the solution $x^\star$, this iteration converges to $x^\star$.  

To extend this solution across a neighbourhood of $c$, we seek a jet approximation $x(\xi)$ satisfying the parametric equation:
$$
f(x(\xi), c + \xi) = 0,
$$
where $\xi$ represents a small perturbation of the parameter $c$. This approach captures the solution's dependence on $c$ through a polynomial expansion.

To compute this solution, we embed polynomial algebra into Newton's method by defining the iteration:
\[
x_{i+1}(\xi) = N(x_i(\xi)) = x_i(\xi) - \frac{f(x_i(\xi), c + \xi)}{f_x(x_i(\xi), c + \xi)},
\]
where all operations are performed in the polynomial algebra framework. Consequently, each iterate $x_{i+1}(\xi)$ remains a polynomial in $\xi$, preserving the jet structure throughout the computation.

The following example illustrates the procedure:
\begin{example}\label{exa:example}
    Given the function $f(x,c)=x^2+x+c$, where $c\in\mathbb{R}$, we want to find a parametric solution of $f(x,c)=0$ around $c=c_0=0$. Notice that $x=0$ is a solution of $f(x,c_0)=0$ $(f(0,0)=0)$.
    The exact solution is given by $x(c)=(-1+\sqrt{1-4c})/2$. The Taylor series of the previous function can be computed by derivation, obtaining
    $$
    x(c_0+\delta_c)=-\delta_c - \delta_c^2 -2\delta_c^3 - 5\delta_c^4 - 14\delta_c^5 -42\delta_c^6 - 132\delta_c^7 - 429\delta_c^8-1430\delta_c^9+ \mathcal O (\delta_c^{10}).
    $$
    
    Instead of using derivation, we can obtain the Taylor series of $x_c$ (up to a given order) applying the Newton Polynomial Technique

    $$
    x_{i+1}=x_i - \dfrac{f(x_i,0+\xi)}{f_x(x,0+\xi)}.
    $$
    
    We look for a functional solution $x(\xi)$ of $f(x(\xi),c(\xi))=0$, where $c(\xi)=c_0+\xi=0+\xi$. We take as initial approximation of the constant solution $x_0(\xi)=0$, since $f(x_0(\xi),c(0))=0$. After iterating the procedure once, we obtain

    $$x_1(\xi)=x_0(\xi) - \dfrac{f(x_0(\xi),0+\xi)}{f_x(x_0(\xi),0+\xi)}=0-\dfrac{0^2+0+0+\xi}{2\cdot0+1}=-\xi.$$
      
    Repeating the procedure and applying the polynomial algebra formulas \eqref{eq:polynomialBasics} to the fraction we get
    
   $$x_2(\xi)=x_1(\xi) - \dfrac{f(x_1(\xi),0+\xi)}{f_x(x_1(\xi),0+\xi)}=0+\dfrac{(-\xi)^2-\xi+0+\xi}{2(-\xi)+1}=-\xi-\xi^2-2\xi^3+\mathcal O(\xi^3).$$
   
   Counting the constant term, the power series of $x_1(\xi)$ has 2 exact terms. After applying the procedure, $x_2(\xi)$ has 4 exact terms. As we will see, at each step of new Newton Polynomial Technique we double the number of exact terms. Indeed, for $x_3(\xi)$ we get

     $$x_3(\xi)=-\xi-\xi ^2 - 2\xi^3 - 5 \xi^4 - 14 \xi^5 -42 \xi^6- 132 \xi^7+ \mathcal O(\xi^{8}). $$
\end{example}

As demonstrated by this example, the Newton procedure applied to jet variables exhibits a quadratic convergence property in the polynomial coefficients: each iteration doubles the number of correct terms in the Taylor expansion. The primary objectives of this paper are:

\begin{enumerate}
    \item To rigorously prove that this order-doubling property holds universally for sufficiently smooth functions under standard regularity conditions.
    \item To study the convergence behaviour when the initial seed $x_0(\xi)$ is not exact. In particular, to study how the error propagates through iterations.
\end{enumerate}

To that end, the structure of the paper is as follows. In Section~\ref{sec:theorem} we give the main result of the paper for exact initial seeds at order 0. There, we prove that at each iterate of the procedure the number of vanishing coefficients of the Taylor expansion of $f(x_i(\xi),c_0+\xi)$ doubles. In Section~\ref{sec:numericalAplication} we expand the previous result to the case when the initial seed contains an error at order 0. In Section~\ref{sec:applications} we give some numerical examples and applications to the procedure. Finally, in Section~\ref{sec:conclusions} we provide conclusions and possible continuations to this work.

\section{Main Theorem}\label{sec:theorem}

In this section we formalize the results discussed up to now. This is the content of Theorem A. Before stating the theorem, we define some preliminary concepts. We start with the concept of polynomial zero of a function $f(x,c)$, defined for $x\in I\subset \R$ and depending on a parameter $c\in\Lambda\subset \R$.
    
    \begin{definition}\label{def:polzero} A polynomial
    $P(\xi)$ is a {\bf polynomial zero up to order $n$} of a function $f(x,c)$ around $(x_0,c_0)\in I\times \Lambda$ if $P(0)=x_0$ and 
    $$f(P(\xi), c_0+\xi)=  \mathcal{O}(\xi^{n+1}).$$
    \end{definition}

In the paper we assume that $f$ is sufficiently smooth (for $x\in I$ and $c\in\Lambda$) and consider pairs $(x_0,c_0)$ such that $x_0$ is a simple zero of $f(x,c_0)$, that is, $f(x_0, c_0)=0$ and $f_{x}(x_0,c_0)=\alpha\neq 0$. For short, we say that $(x_0,c_0)$ is a simple zero $f$ or, directly, that $x_0$ is a simple zero.  This condition is imposed so as to guarantee the existence and unicity of a local solution $x(\xi)$, $x(0)=x_0$, of the functional equation $f(\cdot, c_0+\xi)=0$ for $\xi$ in a neighbourhood of 0. This solution $x(\xi)$ is what we aim to approximate using polynomials. This leads to the concept of polynomial approximation.

\begin{definition}\label{def:polapprox}
Let $x(\xi)$ be a sufficiently smooth function defined on a neighbourhood of $\xi=0$. We say that a polynomial $\tilde{x}_n(\xi)$ is a \textbf{polynomial approximation up to order $n$} if $\tilde{x}_n(\xi)=x(\xi)+\mathcal{O}(\xi^{n+1})$.
    
\end{definition}

The next proposition shows the relation between polynomial approximations and polynomial zeros. 

\begin{prop}\label{prop:approx-zero}
Let $f(x,c)$ be a sufficiently smooth function that depends on a parameter $c\in\mathbb{R}$, $(x_0,c_0)$ be a simple zero of $f$, and $x(\xi)$ be the (unique) functional solution of $f(x(\xi),c_0+\xi)=0$ with $x(0)=x_0$, defined on a neighbourhood of $\xi=0$. Then, $\tilde{x}_n(\xi)$ is a polynomial approximation up to order $n$ of $x(\xi)$ if, and only if, $\tilde x_n (\xi)$ is a polynomial zero up to order $n$ of $f(x, c)$ around $(x_0,c_0)$.
\end{prop}
\begin{proof}

We first prove that polynomial approximation implies polynomial zero.
We have $\tilde x_n(\xi)= x(\xi) + \mathcal O(\xi ^{n+1})$. Applying this to $f(\cdot,c_0+\xi)=0$ and expanding by Taylor with respect to $x$ around $x(\xi)$ we obtain:
$$f(\tilde x_n(\xi),c_0+\xi) = f(x(\xi),c_0+\xi) + f_x(x(\xi),c_0+\xi) \mathcal O(\xi^{n+1}) + \mathcal O(\xi^{2n+2})=\mathcal O(\xi^{n+1}).$$

Therefore, $\tilde x_n(\xi)$ is a polynomial zero up to order $n$ of $f(x,c)$ around $(x_0,c_0)$.

 We now prove that polynomial zero implies polynomial approximation.
    Let $\tilde x_n(\xi)$ be the truncation up to order $n$ of the Taylor polynomial of the function $x(\xi)$. Assume $\tilde y(\xi)$ is not the truncation of order $n$. Then $\exists i$ minimal, $i<n$, such that: 
   $$\tilde y(\xi)=\tilde x(\xi)+ a_i\xi^i+\mathcal O(\xi^{i+1}).$$
   Therefore, $\tilde y(\xi)$ is not the approximation of order $n$ of $x(\xi)$ since the Taylor polynomial is unique.
    Then, using the Taylor expansion of $f(\cdot,c_0+\xi)$ with respect $x$ around $\tilde x_n(\xi)$ we get:
    $$
    \begin{array}{rcl}
        f(\tilde y(\xi),c_0+\xi)&=& f(\tilde x_n(\xi)+a_i\xi^i+\mathcal O(\xi^{i+1}),c_0+\xi) \\
        &=&f(\tilde x_n(\xi), c_0+\xi)+ f_x(\tilde x_n(\xi), c_0+\xi) a_i\xi^i+\mathcal O(\xi^{i+1})\\
        &=&\mathcal O(\xi^{n+1}) + (f_x(x_0,c_0)+\mathcal O(\xi)) a_i \xi^i+\mathcal O(\xi^{i+1})\\
        &=&f_x(x_0,c_0) a_i \xi^i + \mathcal O(\xi^{i+1}).
    \end{array}
    $$
    Hence, $\tilde y(\xi)$ is not a polynomial zero up to order $n$. Therefore, if $\tilde x(\xi)$ is a polynomial zero, it must be the polynomial approximation of $x(\xi)$ up to order $n$.
\end{proof}

 The next result describes how Newton method acts locally when applied using polynomial algebras. The result requires $(x_0,c_0)$ to be a simple zero of $f(x,c)$. As we may see numerically later, even though Newton's method converges linearly in $\R$ in the case of double zeros, it may not converge for polynomial approximations.

        \begin{thmA}\label{teo:main}
    Let $f(x,c)$ be a sufficiently smooth function defined for $x\in I\subset \R$ and $c\in\Lambda \subset\R$. Let $(x_0,c_0)\in I\times \Lambda$ be a simple zero of $f(x,c)$. Let $N(x)$ denote the Newton method using polynomial iterates,
    $$P_{i+1}(\xi)=N(P_i(\xi))=x-\dfrac{f(P_k(\xi),c_0+\xi)}{f_x(P_i(\xi),c_0+\xi)},$$
    
    \noindent and let $P_0(\xi)=x_0$.
    Then, the iterate $P_i(\xi)$ is a polynomial zero up to order $2^i-1$ of $f(x,c)$ around $(x_0,c_0)$. 
    More generally, if $P(\xi)$ with $P(0)=x_0$ is a polynomial zero up to order $n-1$ of $f(x,c)$ around $(x_0,c_0)$, then $N(P(\xi))$ is a polynomial zero up to order $2n-1$. 
    \end{thmA}

    \proof
        Notice that the (constant) polynomial $P_0(\xi)=x_0$ is a polynomial zero of order $0=2^0-1$. Indeed, by hypothesis we have that $f(x_0,c_0+\xi)=0+\mathcal{O}(\xi)$.
        
        For the sake of clarity, next we prove that $P_1$ is a polynomial zero up to order $2^1-1$. We have:
        \begin{align*}
                P_1(\xi)&= P_0(\xi) - \dfrac{f(P_0(\xi), c_0+\xi)}{f_x(P_0(\xi),c_0+\xi)}\\
                        &= x_0 - \dfrac{f(x_0, c_0+\xi)}{f_x(x_0,c_0+\xi)}.
        \end{align*}
   
Evaluating $f$ using its Taylor expansion with respect to $x$ around $x_0$ we get:
    \begin{align*}
    f(P_1(\xi),c_0+\xi) =& f\left(x_0 - \dfrac{f(x_0, c_0+\xi)}{f_x(x_0,c_0+\xi)},c_0+\xi \right)    \\
    =& f(x_0,c_0+\xi) +f_x(x_0,c_0+\xi) \left( -\dfrac{f(x_0, c_0+\xi)}{f_x(x_0,c_0+\xi)}\right)\\& + \frac{1}{2}f_{xx}(x_0,c_0+\xi) \left(- \dfrac{f(x_0, c_0+\xi)}{f_x(x_0,c_0+\xi)}\right)^2+\mathcal{O}\left(\left(- \dfrac{f(x_0, c_0+\xi)}{f_x(x_0,c_0+\xi)}\right)^3\right)\\
     =& \frac{1}{2}f_{xx}(x_0,c_0+\xi) \left(- \dfrac{f(x_0, c_0+\xi)}{f_x(x_0,c_0+\xi)}\right)^2+\mathcal{O}\left(\left(- \dfrac{f(x_0, c_0+\xi)}{f_x(x_0,c_0+\xi)}\right)^3\right).
    \end{align*}
     
     Observe that, since  $f=\mathcal O(\xi)$ and $g=\mathcal O(1)$, $f/g=\mathcal O(\xi)$ (see Lemma~\ref{lem:ordrediv}). Since $(x_0, c_0)$ is a simple zero of $f(x,c)$, there is $\alpha\neq 0$ such that $f_x(x_0,c_0+\xi)=\alpha+\mathcal O(\xi)$. We get that
    $$f(P_1(\xi),c_0+\xi)=0+\mathcal{O}(\xi^2).$$
    We conclude that $P_1(\xi)$ is a polynomial zero of order $1=2^1-1$.

    We now prove that if $P(\xi)$ is a polynomial zero up to order $n-1$ of $f(x,c)$ around $(x_0,c_0)$, then $N(P(\xi))$ is a polynomial zero up to order $2n-1$. We have
    
    $$f(P(\xi),c_0+\xi)=\mathcal{O}\left(\xi^{n}\right).$$
    
    The polynomial $N(P(\xi))$ is given by
    $$N(P(\xi))= P(\xi) - \underbrace{\dfrac{f(P,c_0+\xi)}{f_x(P,c_0+\xi)}}_{\widetilde P(\xi)}.$$
    
    Evaluating $f(\cdot,c_0+\xi)$ at $N(P(\xi))$ and using the Taylor expansion of $f(x,c_0+\xi)$ with respect to $x$ around $P(\xi)$ we get
    \begin{align*}
        f(N(P(\xi)),c_0+\xi)=& f\left(P(\xi) - \tilde P(\xi),c_0+ \xi\right) \\
                        =& f(P(\xi),c_0+\xi) - f_x(P(\xi),c_0+\xi) \dfrac{f(P(\xi),c_0+\xi)}{f_x(P,c_0+\xi)} \\
                         &+\frac 1 2 f_{xx}(P(\xi),c_0+\xi) (-\widetilde P(\xi))^2+\mathcal{O}\left(\left(\tilde P(\xi)\right)^3\right).
    \end{align*}

    We have $f(P(\xi),c_0+\xi)=\mathcal{O}\left(\xi^{n}\right)$. Since $(x_0, c_0)$ is a simple zero of $f(x,c)$, there is $\alpha\neq 0$ such that $f_x(P(\xi),c_0+\xi)=\alpha+\mathcal O(\xi)$. By Lemma~\ref{lem:ordrediv} we get
    
    $$
    \tilde P(\xi)=\dfrac{f(P(\xi),c_0+\xi)}{f_x(P(\xi),c_0+\xi)}= \mathcal{O}\left(\xi ^{n}\right).
    $$
    Finally, since $f_{xx}(P(\xi),c_0+\xi)=\mathcal O(1)$, we get
    \begin{align*}
        f(N(P(\xi)),c_0+\xi)
                        =&\dfrac 1 2  f_{xx}(P(\xi),c_0+\xi) \left(\tilde P(\xi)\right)^2+\mathcal{O}\left(\left(\tilde P(\xi)\right)^3\right)\\
                        =& \dfrac 1 2 f_{xx}(P(\xi),c_0+\xi) \left({\mathcal{O}\left(\xi^{n}\right)}\right)^2+\mathcal{O}\left(\left(\mathcal{O}\left(\xi^{n}\right)\right)^3\right)\\
                        =& \mathcal{O}\left(\xi^{2n}\right).
    \end{align*}

    Therefore, $N(P(\xi))$ is a polynomial zero up to order $2n-1$ of $f$. The fact that $P_i$ is a polynomial zero up to order $2^i-1$ follows by induction on $i$ from the previous result using that $P_1(\xi)$ is a polynomial zero up to order $2^1-1$ and that $P_{i+1}(\xi)=N(P_i(\xi))$.
   
\endproof

Combining Theorem A with Proposition~\ref{prop:approx-zero} we obtain the following result. 

\begin{corollary}\label{cor:double}
    Under the assumptions of Proposition~\ref{prop:approx-zero} and Theorem A, if $P(\xi)$ is a polynomial approximation of $x(\xi)$ of order $n-1$, then $N(P(\xi))$ is a polynomial approximation of $x(\xi)$ of order $2n-1$.
\end{corollary}
\begin{remark}
    A polynomial approximation of order $n-1$ contains $n$ correct coefficients including the term of order 0. Thus, Corollary \ref{cor:double} implies that the number of correct coefficients is doubled at each step of Newton iteration.
\end{remark}

\section{Numerical application of the method}\label{sec:numericalAplication}

In this section, we deal with the practical implementation of a polynomial Newton method. In any implementation, the coefficients of the polynomial $\tilde x(\xi)$ will be stored as floating point numbers. Therefore, from a practical point of view, there will be no exact solutions of the polynomial equation, but solutions with a given error $\varepsilon$. The following definition introduces this concept.

\begin{definition}
    A polynomial $P(\xi)$ is a {\bf polynomial zero of error $\varepsilon$ up to order $n$} of a function $F$ if 
    $$F(P(\xi))=\sum_{i=0}^n a_i(\varepsilon) \xi^i + \mathcal O(\xi^{n+1}),$$
    where $a_i(\varepsilon)=\mathcal O(\varepsilon)$.
\end{definition}

For the sake of completeness let us define also the equivalent concept of a polynomial approximation up to order $n$ with some error.
\begin{definition}\label{def:polapproxError}
Let $x(\xi)$ be a sufficiently smooth function defined on a neighbourhood of $\xi=0$. We say that a polynomial $\tilde{x}_n(\xi)$ is a \textbf{polynomial approximation of error $\varepsilon$ up to order $n$} if $\tilde{x}_n(\xi)=x(\xi)+\sum_{i=0}^n b_i\xi^i+\mathcal{O}(\xi^{n+1}),$ where $b_i=\mathcal O(\varepsilon)$ for all $i=0,\dots,n$.
    
\end{definition}

The following proposition provides the equivalent to Proposition~\ref{prop:approx-zero} for the inexact case. The proof is analogous to the one of Proposition~\ref{prop:approx-zero}.

\begin{prop}\label{prop:approx-zeroError}
Let $f(x,c)$ be a sufficiently smooth function that depends a parameter $c\in\mathbb{R}$, $(x_0,c_0)$ be a simple zero of $f$, and $x(\xi)$ be the (unique) functional solution of $f(x(\xi),c_0+\xi)=0$ with $x(0)=x_0$, defined on a neighbourhood of $\xi=0$. Then, $\tilde{x}_n(\xi)$ is a polynomial approximation of error $\varepsilon$ up to order $n$ of $x(\xi)$ if, and only if, $\tilde x_n (\xi)$ is a polynomial zero of error $\varepsilon$ up to order $n$ of $f(x, c)$ around $(x_0,c_0)$.
\end{prop}

The following remark expands the evaluation of $f(\tilde x(\xi),c_0+\xi)^k$. It will be useful in the proof of Theorem~\ref{teo:maineps}, which is the analogous to  Theorem~A in the inexact case.

\begin{remark} \label{rmk:poly0orNeps}
    Observe that, the $k$-th power of $f$ evaluated at a polynomial zero of error $\varepsilon$ up to order $n$, $\tilde x(\xi)$, verifies:
    $$
    \begin{array}{rcl}
    f(\tilde x(\xi),c_0+\xi)^k&=&\displaystyle \left( \sum_{i=0}^n a_i(\varepsilon) \xi^i + \mathcal O(\xi^{n+1})\right)^k=\\
    &=&\displaystyle  \left( \sum_{i=0}^n a_i(\varepsilon) \xi^i\right)^k + k  \left( \sum_{i=0}^n a_i(\varepsilon) \xi^i\right)^{k-1}\mathcal O(\xi^{n+1}) \\
    &&\displaystyle + \sum _{j=2}^k \begin{pmatrix}k\\j\end{pmatrix}\left( \left( \sum_{i=0}^n a_i(\varepsilon) \xi^i\right)^{k-j} \mathcal O(\xi^{n+1})^j\right).
    \end{array}
    $$
    
    Hence, all the elements in the first term contain a $k$-tuple of products of $a_i(\epsilon)$. Therefore, all the coefficients will be of order $\mathcal O(\varepsilon^k)$. For the elements on the second term, they will have $(k-1)$-tuples of products. Hence, the coefficients are of order $\mathcal O(\varepsilon^{k-1})$. In addition, the smallest order in $\xi$ that appear is $n+1$. From that point on, the smallest terms in $\xi$ that appear are of order $2n+2$.
\end{remark}

The following theorem explores how the errors propagate through the coefficients of polynomial evaluation. In particular, the polynomial Newton method behaves quadratically in the error and in the number of coefficients:

\begin{theorem}\label{teo:maineps}
Let $f(x,c)$ be a sufficiently smooth function depending on a parameter $c\in\mathbb{R}$. Let $(x_0,c_0)$ be a simple zero of $f(x,c)$.    If $\tilde x(\xi)$ is a polynomial zero of error $\varepsilon$ up to order $n$ of a function $f(\cdot, c_0+\xi)$ such that $\tilde x(0)=x_0+\mathcal O(\varepsilon)$, then:
    \begin{enumerate}
        \item $N(\tilde x(\xi))$ is a polynomial zero of error $\varepsilon^2$ up to order $n$  for $f(\cdot, c_0+\xi)$.
        \item $N(\tilde x(\xi))$ is a polynomial zero of error $\varepsilon$ up to order $2n+1$ for $f(\cdot, c_0+\xi)$.
    \end{enumerate}
\end{theorem}

\begin{proof}
    Let $\tilde x(\xi)$ be a polynomial zero of error $\varepsilon$ up to order $n$. Then: 
    $$f(\tilde x(\xi),c_0+\xi) = \sum_{i=0}^n a_i(\varepsilon) \xi^i + \mathcal O(\xi^{n+1}).$$
    We consider the next iterate trough the Jet Newton procedure:
    $$\tilde y(\xi)=N(\tilde x(\xi))=\tilde x(\xi)- \dfrac{f(\tilde x(\xi),c_0+\xi)}{f_x(\tilde x(\xi),c_0+\xi)}.$$
    Evaluating in $f(\cdot,c_0+\xi)$ and taking the Taylor expansion around $(\tilde x(\xi),c_0+\xi)$ it gives:
    $$
    \begin{array}{rcl}
    f(\tilde y(\xi),c_0+\xi)&=&f\left(\tilde x(\xi)- \dfrac{f(\tilde x(\xi),c_0+\xi)}{f_x(\tilde x(\xi),c_0+\xi)},c_0+\xi \right)\\
    &=&f(\tilde x(\xi),c_0+\xi) + f_x(\tilde x(\xi),c_0+\xi) \dfrac{-f(\tilde x(\xi),c_0+\xi)}{f_x(\tilde x(\xi),c_0+\xi)} \\
    &&+ \dfrac{1}{2}f_{xx}(\tilde x(\xi),c_0+\xi) \left(\dfrac{-f(\tilde x(\xi),c_0+\xi)}{f_x(\tilde x(\xi),c_0+\xi)}\right)^2 + \mathcal O\left(\left(\dfrac{-f(\tilde x(\xi),c_0+\xi)}{f_x(\tilde x(\xi),c_0+\xi)}\right)^3\right)\\
    &=&\dfrac{f_{xx}(\tilde x(\xi),c_0+\xi)}{2f_x(\tilde x(\xi),c_0+\xi)^2} f(\tilde x(\xi),c_0+\xi)^2 + \mathcal O\left(\left(\dfrac{-f(\tilde x(\xi),c_0+\xi)}{f_x(\tilde x(\xi),c_0+\xi)}\right)^3\right).
    \end{array}
    $$

    The term $f_{xx}(\tilde x(\xi),c_0+\xi)/2f_x(\tilde x(\xi),c_0+\xi)^2$ can be written as $q(\xi):=e_0+\sum_{i=1}^\infty e_i\xi^i$,  where the coefficients $e_i$ depend on the derivatives of $f$ around $(x_0,c_0)$. 
    Then, $ f(\tilde y(\xi),c_0+\xi)$ can be written as

    \begin{eqnarray}
    f(\tilde y(\xi),c_0+\xi)&=&\displaystyle q(\xi) \left( \sum_{i=0}^n a_i(\varepsilon) \xi^i + \mathcal O(\xi^{n+1})\right)^2 + \mathcal O\left(\left(\dfrac{-f(\tilde x(\xi),c_0+\xi)}{f_x(\tilde x(\xi),c_0+\xi)}\right)^3\right) \nonumber\\
    &=& \displaystyle  q(\xi) \left(\left( \sum_{i=0}^n a_i(\varepsilon) \xi^i\right)^2 + 2\left( \sum_{i=0}^n a_i(\varepsilon) \xi^i\right) \mathcal O(\xi^{n+1}) + \left( \mathcal O(\xi^{n+1})\right)^2\right) \label{eq:firsteps}
    \\&&+ \mathcal O\left(\left(\dfrac{-f(\tilde x(\xi),c_0+\xi)}{f_x(\tilde x(\xi),c_0+\xi)}\right)^3\right). \label{eq:secondeps}
    \end{eqnarray}

Observe that, those terms that do not contain $\mathcal O(\varepsilon^2)$ in \eqref{eq:firsteps} are 
    $$\displaystyle  \left(e_0+\sum_{i=1}^\infty e_i\xi^i \right) \left( 2 \sum_{i=0}^n a_i(\varepsilon) \xi^{i+n+1} 
        + \left( \mathcal O(\xi^{2n+2})\right)\right).
        $$
From the higher order terms in \eqref{eq:secondeps}, by Remark~\ref{rmk:poly0orNeps}, the terms which are not of the form $\mathcal O(\varepsilon^2)$ are of order higher than $n+1$. Therefore, the lowest order in $\xi$ whose coefficient is not of the form $\mathcal O(\varepsilon^2)$ is $n+1$. This proves the first statement.

Analogously, those terms that are not of the form  $\mathcal O(\varepsilon)$ are given by
$$\displaystyle  \left(e_0+\sum_{i=1}^\infty e_i\xi^i \right) \left(  \mathcal O(\xi^{2n+2})\right).$$
The higher order terms in \eqref{eq:secondeps}, by Remark~\ref{rmk:poly0orNeps}, are of order greater than $2n+2$.
Therefore, the first $2n+1$ terms are at least of order $\mathcal O(\epsilon)$. This proves the second statement.
\end{proof}

\begin{remark}
    As a corollary of Theorem A we had that at each step we doubled the number of exact coefficients. Analogously, it follows from  Theorem~\ref{teo:maineps} (and Proposition~\ref{prop:approx-zeroError}) that at each step we double the number of terms whose coefficients have error of the form $\mathcal O(\epsilon)$. Indeed, if $\tilde x(\xi)$ is a polynomial zero of error $\epsilon$ up to order $n$ then, counting the therm of order zero, it has $n+1$ terms with error $\mathcal O(\epsilon)$ while $N(\tilde x(\xi))$ has $2n+2$ terms with error $\mathcal O(\epsilon)$.
\end{remark}

\section{Applications }\label{sec:applications}

To illustrate the previous algorithms and the examples shown we numerically tested the results on different nonlinear equations. Preliminary tests have shown that, due to double precision approximation errors, the coefficients of order higher than $40$ are not relevant: the errors of double precision coefficients are higher than the information provided by those terms. For that reason, at all the numerical applications we use polynomial approximations up to order $40$. 

The code used in this section can be found at \cite{soft-NewtonJet}. It has been written from scratch with the exception of the polynomial algebra that has been generated using the package Taylor2 from \cite{GimenoJZ22}.

\subsection{An academic example}
In the introduction we presented a first academic example,  Example~\ref{exa:example}.
Next, we will study the numerical estimates for this example. Recall that we consider the second degree equation $f(x,c)=x^2+x+c$, where $c\in\mathbb R$. We want to obtain the solution $x(c)$ such that $f(x(c),c)=0$ in a neighbourhood of $c=c_0$.

This example is interesting since we already know the exact solution $x(c)=\dfrac{-1+\sqrt{1-4c}}{2}$ and its Taylor approximation around $c_0=0$ is:
    $$
    x(c_0+\delta_c)=\sum_{i\geq 0} x_{[i]} \delta_c^i =-\delta_c - \delta_c^2 -2\delta_c^3 - 5\delta_c^4 - 14\delta_c^5 -42\delta_c^6 - 132\delta_c^7 - 429\delta_c^8-1430\delta_c^9+ \mathcal O (\delta_c^{10}),
    $$
where $\delta_c\in \mathbb R$ denotes the increment with respect to the parameter $c_0$ and $x_{[i]}$ denotes the coefficient of $\delta_c^i$.
We start the method with the initial polynomial $P_0(\xi)=x_0$, where $x_0\in \mathbb R$ is an initial approximation to the solution of $f(x,c_0)=0$. After applying the Newton algorithm $n$ times we get the polynomials $P_n(\xi)=\sum_{i=0}^N p_{n,[i]} \xi^i$. Each iterate of the Newton method is given by:
$$P_{n+1}(\xi)= P_n(\xi) - \dfrac{P_n(\xi)^2+ P_n(\xi) + \xi}{2\cdot P_n(\xi) + 1}.$$
Recall that those computations are done using polynomial arithmetic.
We will consider two different values for $x_0$. We first consider $x_0=0$, which corresponds to the exact solution at $c=0$. Therefore, our initial polynomial is $P_0(\xi)=0$. We then apply Newton's algorithm to jets. The left plot of Figure~\ref{fig:academicjets} shows the evolution of the logarithm of the error of the coefficients with respect to the correct values. Since the constant term in the Taylor's expansion $x(c_0+\delta_c)$ is 0 we will use the absolute error, i.e. $|x_{[0]}-p_{n,[0]}|$. For all the other terms the figures shows relative errors, i.e. $|x_{[i]}-p_{n,[i]}|/|x_{[i]}|$. In general terms, along the paper we use the absolute error whenever the value of the correct value is below 1 and the relative one when it's over 1. The relative error was chosen for this comparison because both values possess large integer parts. Consequently, the absolute error is dominated by the difference in these integers, obscuring the convergence of the significant digits and yielding a misleadingly high value.
This serves as a measure of the polynomial expression's accuracy at each iteration. Red values indicate errors exceeding the coefficient's magnitude, while black values represent errors below double floating-point precision. The colour scale is fixed such that each Newton iteration advances one colour level. The plot illustrates how at each iteration the number of correct coefficients is doubled. Afterwards, we consider as initial approximation $x_0=0.1$. This is done to assess the numerical stability of the method. The right plot of Figure~\ref{fig:academicjets} shows the evolution for $P_0(\xi) =x_0=0.1$ (an initially perturbed polynomial solution). We observe the expected quadratic convergence behaviour (error squaring) in each column. Additionally, we obtain twice as many coefficients with comparable error to the previous iteration.

\begin{figure}
    \centering
    \includegraphics[width=0.49\linewidth]{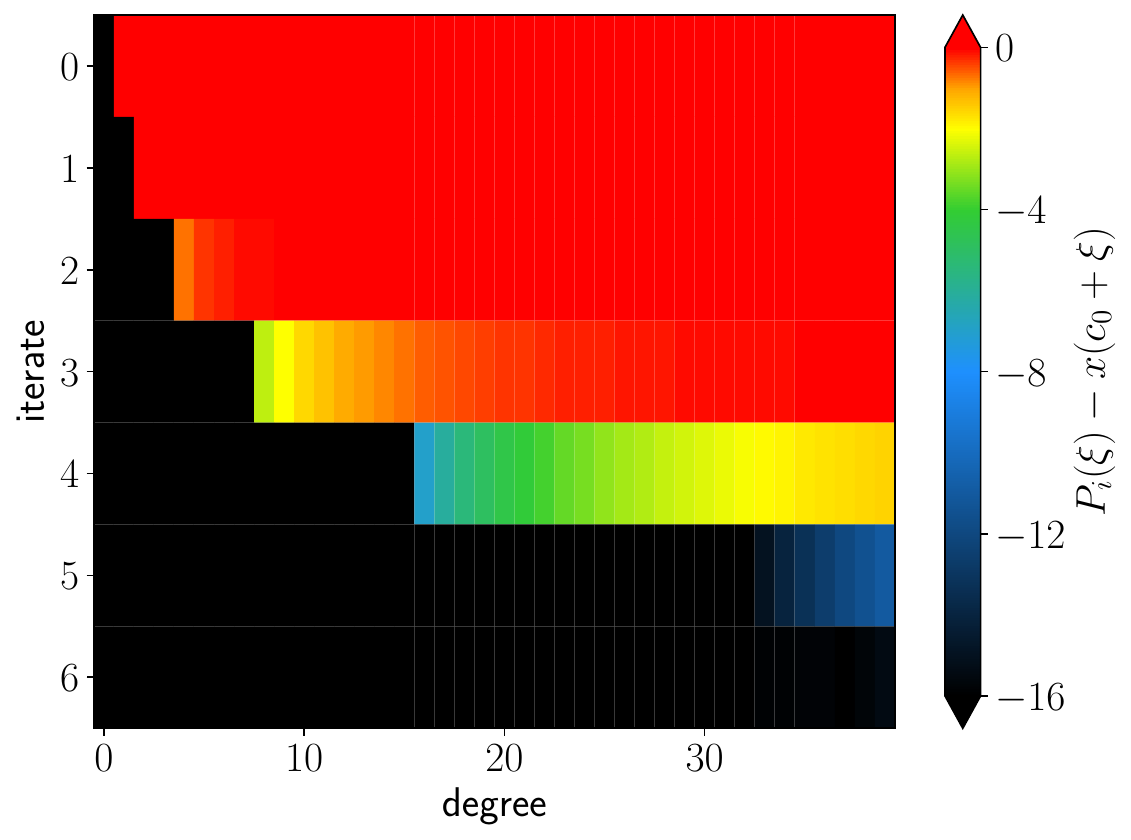}
    \includegraphics[width=0.49\linewidth]{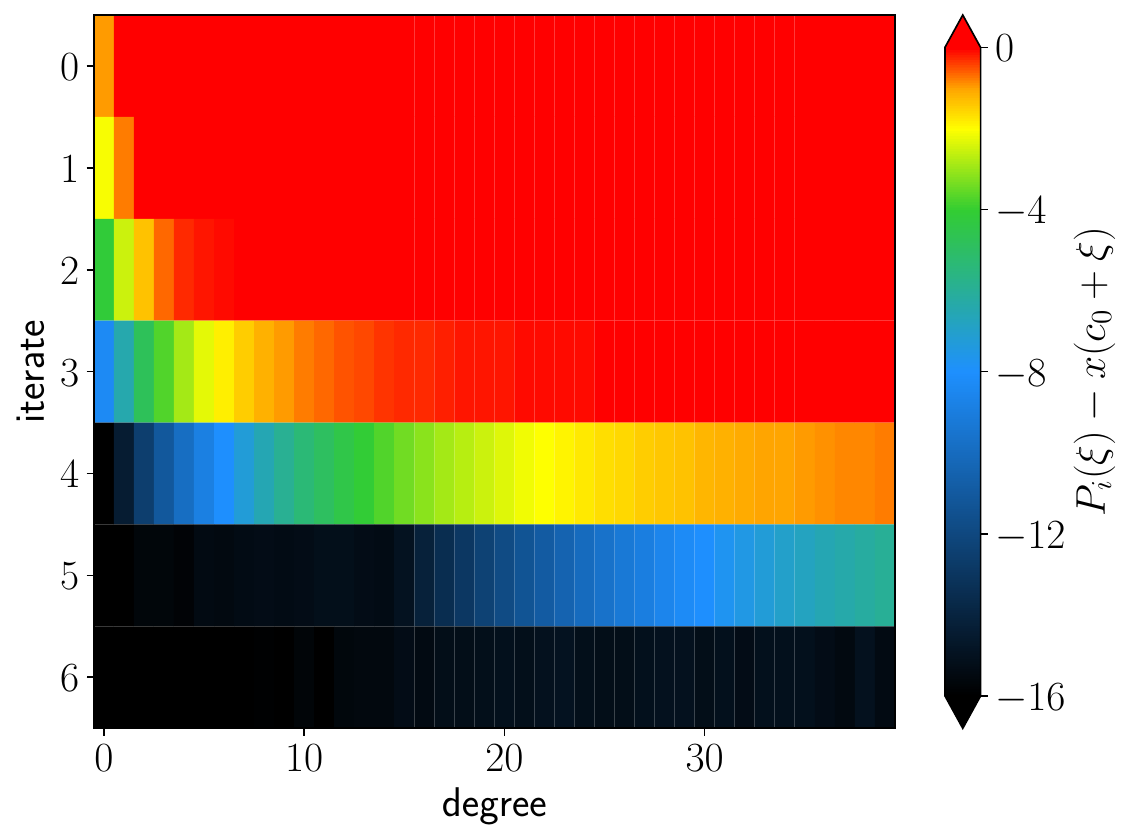}
    \caption{Error evolution of the coefficients for the initial condition $x_0=0$ (left) and $x_0=0.1$ (right) for the quadratic equation example. Each row represents the iterate ($i$) while each column the degree of the coefficient. Absolute error is used for degree 0 and relative error is used for all the other degrees.}
    \label{fig:academicjets}
\end{figure}

In order to assess the numerical application, we evaluate the range of values $\xi$ where the polynomial is correct (up to a given tolerance). In Figure~\ref{fig:academic_evalP} we plot in purple the result of evaluating the polynomial $P_6(\xi)$ compared against the result of $x(\xi)$. Notice from Figure~\ref{fig:academicjets} that all the coefficients of $P_6(\xi)$ up to order 40 are correct up to the double precision level independently of whether we take $P_0(\xi)=0$ or $P_0(\xi)=0.1$. Nonetheless, for this figure, we use $P_0(\xi)=0$. The difference between $P_6(\xi)$ and $x(\xi)$ is shown in decimal logarithm base. In green we show the logarithm of evaluating $P_6(\xi)$ at the function, i.e. $f(P_6(\xi),c_0+\xi)$, where $c_0=0$. This is done by using a grid of points in the interval $[-0.25,0.25]$. Notice that the theoretical solution is $x(c)=(-1+\sqrt{1-4c})/2$, so the radius of convergence of the Taylor series cannot go beyond $c=1/4$.
Within the plot, we observe how the error grows as we approach such a limit for the radius of convergence of the Taylor series. Nonetheless, this error is also affected by the truncation of the polynomial. Therefore, the plot reveals both the Taylor series' radius of convergence and the truncation effects. We define the \emph{effective radius of convergence} as the maximum $\xi$ for which $|f(P_6(\xi), c_0+\xi)| < \texttt{tol}$, where \texttt{tol} is a prescribed tolerance threshold.

\begin{figure}
    \centering
    \includegraphics[width=\linewidth]{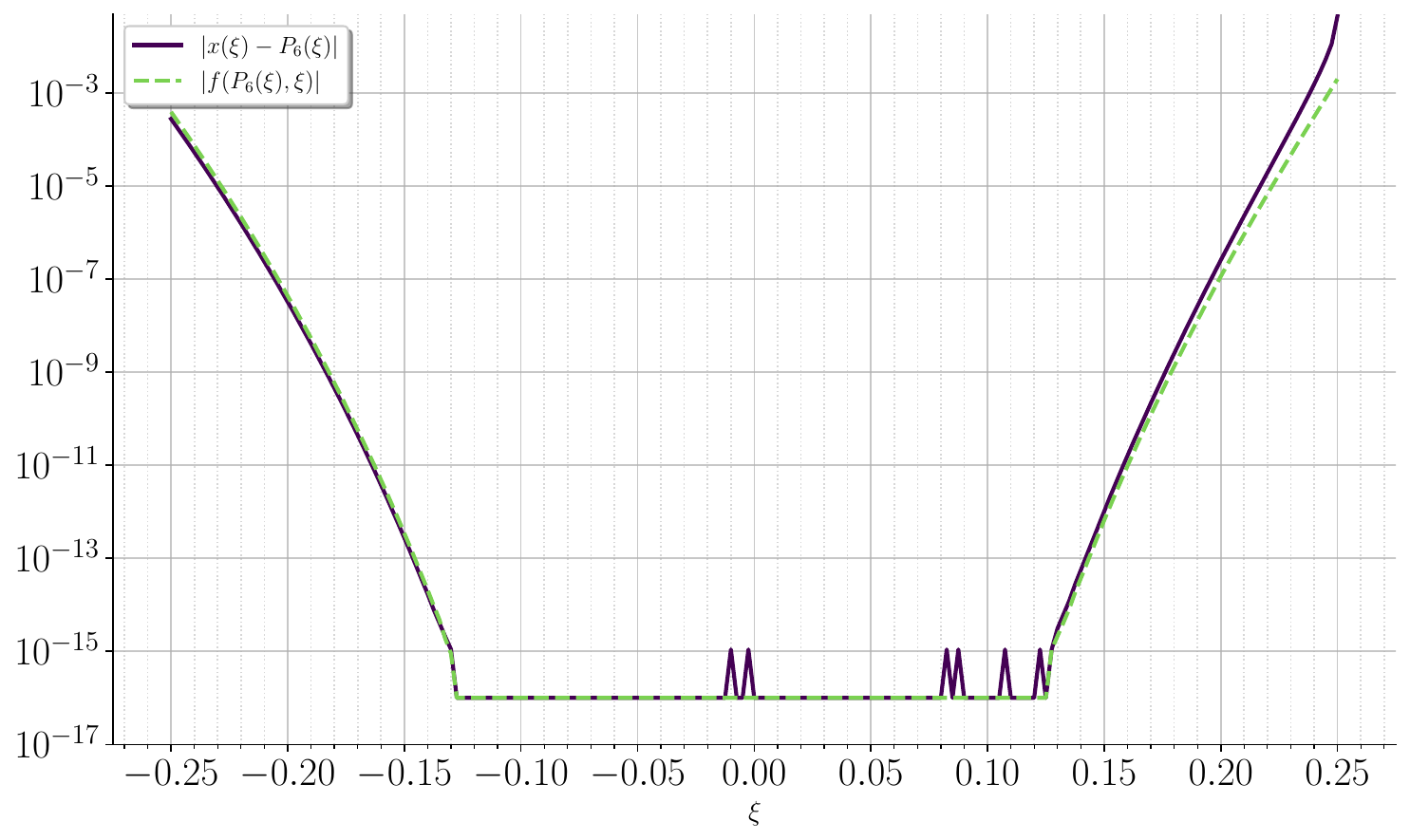}
    \caption{Difference between $P_6(\xi)$ and the function $x(\xi)$ evaluated at values $\xi$ up to the effective radius of convergence of the Polynomial (purple). Evaluation of $f(P(\xi),c_0+\xi)$ up to the effective radius of convergence of the polynomial (green).}
    \label{fig:academic_evalP}
\end{figure}

The final test consists of evaluating the function $f(P_i(\xi),\xi)$ as a jet variable, that is, we consider $f(P_i(\xi),\xi)$ as a polynomial in $\xi$. Theorem~\ref{teo:maineps} states that an initial seed which is a polynomial zero of order $0$ and error $\varepsilon$ generates an $n$-th iterate that is a polynomial zero of order $2^n$ and error $\varepsilon$. Consequently, for an exact initial seed ($\varepsilon=0$), the first $2^n$ coefficients of the $n$-th iterate are below a given tolerance.

Figure~\ref{fig:academic_evalPQuad} (top row) shows the evolution of these coefficients: the left panel corresponds to the initial condition $x_0 = 0$, while the right panel uses $x_0 = 0.1$ (matching the polynomial evaluations in Figure~\ref{fig:academicjets}). As expected, iterations double the number of vanishing coefficients for $x_0 = 0$ and exhibit quadratic convergence for $x_0 = 0.1$. However, the final iterates deviate from this behaviour due to limitations in double floating-point precision. 
The apparent zero error in the high-order Jet coefficients during the initial iterations is an initialization artifact. Since the Jet variables are initialized with zero coefficients, their initial error is necessarily zero. These values are artificial and do not reflect the true convergence of the method. A robust measurement protocol should therefore ignore these artifacts by defining the polynomial error based on the lowest-order coefficient that has an error above a predefined tolerance: when looking at Figure~\ref{fig:academic_evalPQuad}, the validity of a value is given by the worst colour at its left (including itself).

To study the limits due to double-precision arithmetic,  we repeated the test with quadruple precision (bottom row of Figure~\ref{fig:academic_evalPQuad}). The extended colour range (up to $10^{-32}$) reveals that coefficients near degree 32 have errors of $10^{-16}$. This occurs because terms of order $10^{16}$ lose 16 digits to numerical cancellation during arithmetic operations.  
\begin{figure}
    \centering
    \includegraphics[width=0.49\linewidth]{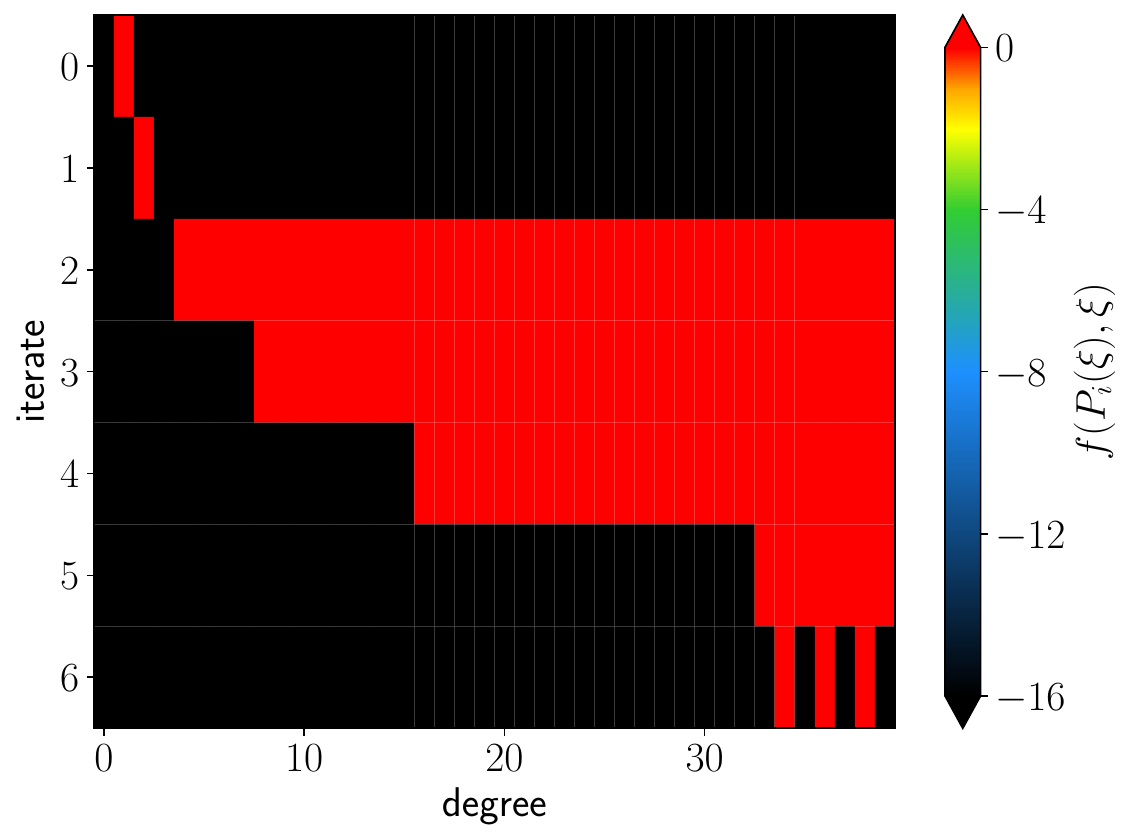}
    \includegraphics[width=0.49\linewidth]{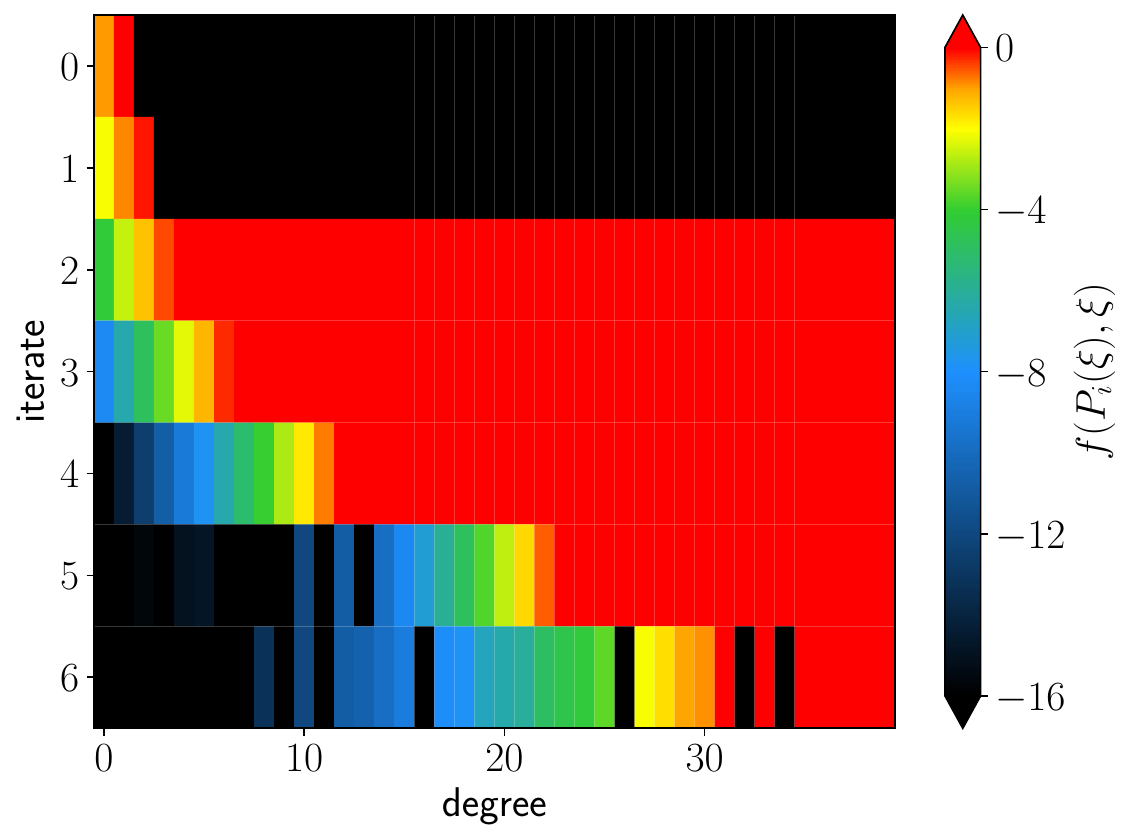}
    \includegraphics[width=0.49\linewidth]{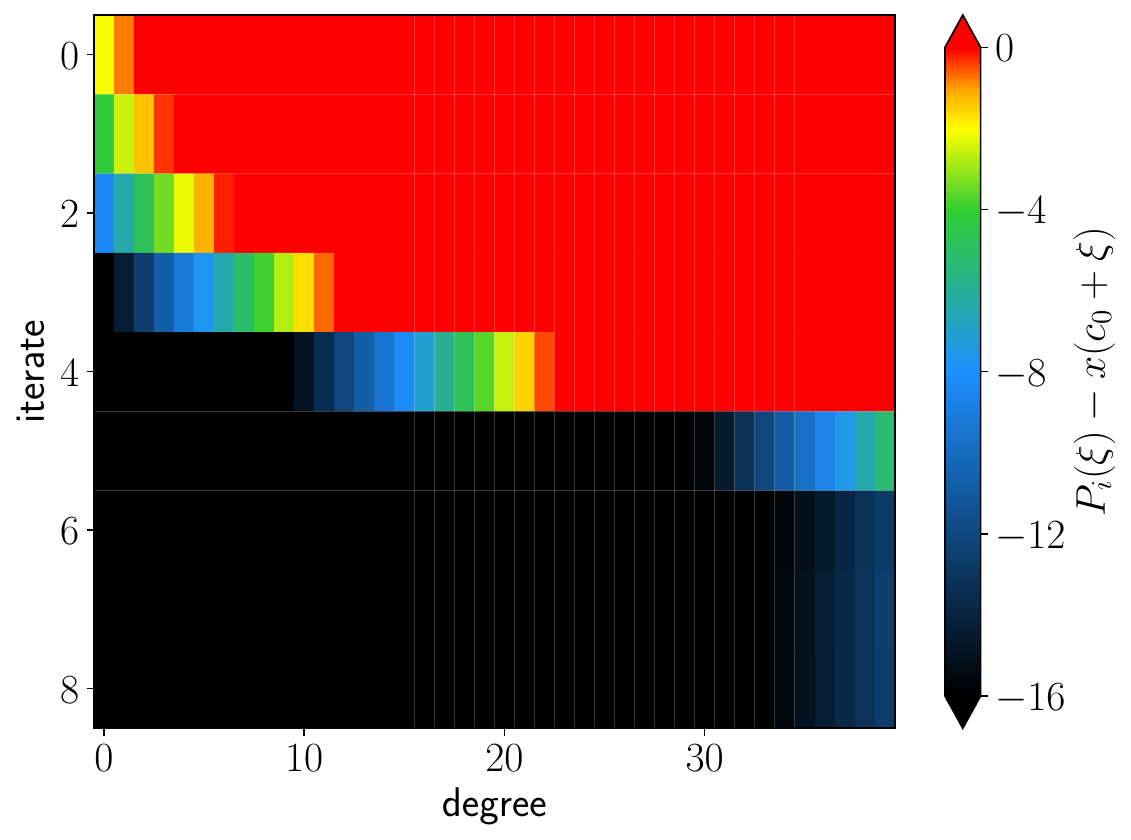}
    \includegraphics[width=0.49\linewidth]{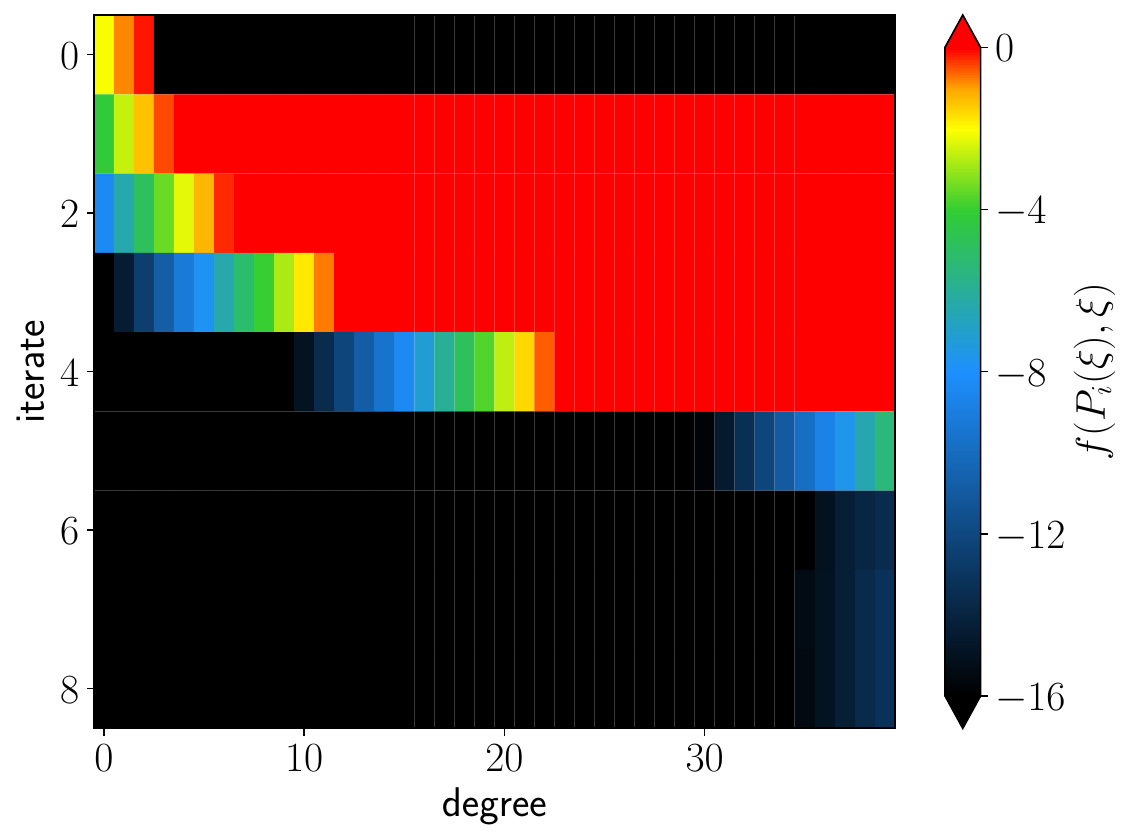}
    \caption{Left: Absolute error of the coefficients of the polynomial $x(\delta c)$ for an initial $x_0=0.1$. Right: coefficients of the evaluation of $f(P(\delta c),\delta c)$. Both computations are done in quadruple precision.}
    \label{fig:academic_evalPQuad}
\end{figure}

\subsection{Kepler equation}
The second example is a well known equation in celestial mechanics: the Kepler's equation \cite{kepler1609}. It relates the mean anomaly ($M$) with the eccentric anomaly ($E$) through the eccentricity ($e$) of a planar elliptic orbit in the context of a two body problem:
$$M=E- e \sin(E).$$
Given a celestial body $P$ (for instance a planet) rotating around a massive body $S$ (for instance the Sun) in an elliptical orbit of eccentricity $e$ and semimajor axis $a$, the eccentric anomaly is defined as the angle between $S$, the center of the ellipse, and the perpendicular projection of $P$ on the circle of centre $O$ and radius $a$. The mean anomaly is the fraction of an elliptical orbit's period that has elapsed since the orbiting body passed the periapsis. Figure~\ref{fig:kepleresEquationRepres} shows the definition of both angles. For more information on the Kepler's equation see \cite{Danby88}.

\begin{figure}
    \centering
       \includegraphics[width=0.5\linewidth]{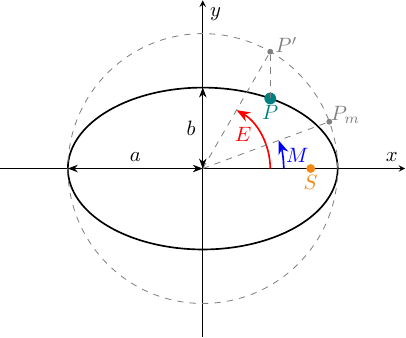}
    \caption{Scheme of the elements contained in Kepler's equation.}
    \label{fig:kepleresEquationRepres}
\end{figure}

It is easy to get the mean anomaly once the eccentric anomaly is known. However,  to obtain the eccentric anomaly from the mean anomaly is not trivial. In order to get the value of $E$ for different values of $M$ the following equation is considered:
$$
f_K(E, M)= E- e\sin(E) - M=0.
$$
Therefore, solving $f_K(E,M_0)=0$ for a given mean anomaly $M_0$ determines the value of the eccentric anomaly $E$. Using the Jet transport in the Newton Method we can obtain an approximation of the expression of $E(\xi)$ such that $f_K(E(\xi),M_0+\xi)=0$ for a given value of $M_0$. To do so, we consider an initial condition $E_0=\varepsilon$, $\varepsilon=0$ and $\varepsilon=0.1$,  and apply Newton's-Jet method:
$$
E_i= E_i - \dfrac{f_K(E_i,M_0+\xi)}{f'_K(E_i,M_0+\xi)}.
$$

Figure~\ref{fig:keplerCoefs} shows the convergence of the coefficients of the polynomial $f_K(E_i(\xi),M_0+\xi)$. In that case, we cannot use the value of the exact function $E(\xi)$ to validate the results since it is not known. Observe that since the equation that we want to solve contains only odd terms, the solution only contains odd terms. Therefore, at each Newton step, we advance a an additional degree in the Jet of the solution. The left plot shows the evolution of the method for the initial condition $E_0=0$ which for $M_0=0$ gives a solution of Kepler's equation. The right plot shows the evolution for the initial condition $E_0=0.1$. Therefore, we observe the convergence to the solution of the independent term. Once the independent term converges, the higher order terms in $\xi$ converge as expected. The fact that the solution only contains odd terms is also reflected, although the convergence is slower than in the exact initial condition (left plot). As shown in Figure~\ref{fig:academic_evalPQuad}, the initial iterate exhibits low error for both the initial value and high-order terms. This is, again, an initialization artifact. The validity of the convergence of the Jet must be assessed from the lowest-order term whose error exceeds a given tolerance. For instance, at the first iterate in the figure, the highest error is in the third-order term. This error invalidates the accuracy of all higher-order coefficients, as the solution cannot be considered correct at a higher order if it is incorrect at a lower one.

\begin{figure}
    \centering
    \includegraphics[width=0.49\linewidth]{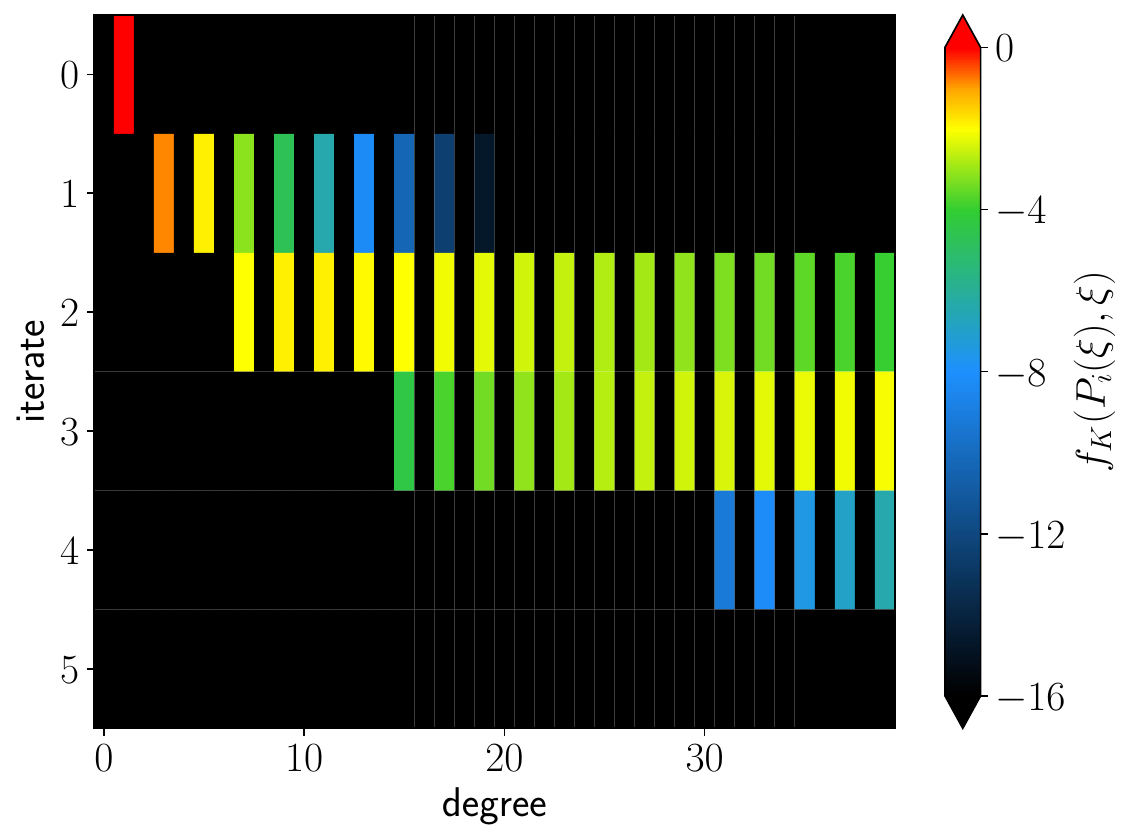}
    \includegraphics[width=0.49\linewidth]{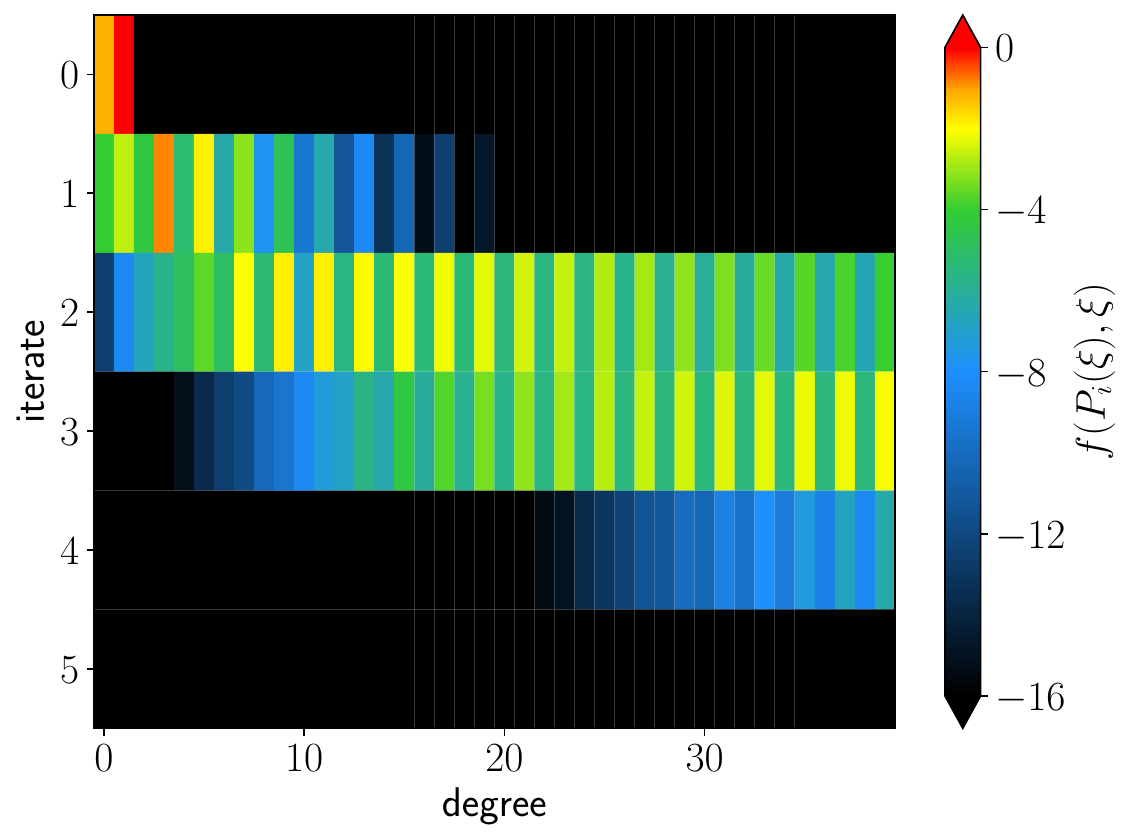}
    \caption{Taylor series coefficients of the Kepler function \( f_K(E(\xi),\xi) \), propagated using Jet transport. \textbf{Left:} Coefficients starting from the initial condition \( E_0 = 0 \). \textbf{Right:} Coefficients starting from \( E_0 = 0.1 \). The mean anomaly is fixed at \( M_0 = 0 \) in both cases.}
    \label{fig:keplerCoefs}
\end{figure}

\subsection{Zero Velocity curves in the Restricted Three Body Problem} \label{sec:zerovelocity}
The circular restricted three body problem (CR3BP) is a classical problem in celestial mechanics (see \cite{szebehely_theory_1967,kolomaro}). It studies the movement of a massless object subject to the gravitational attraction of two massive bodies called primary bodies. Both massive bodies move in circular orbits around their gravity centre. In order to simplify the model, the system is considered in rotating coordinates and, hence, the positions of the primaries are fixed. In addition, the unit of distance is defined in such way that the distance between the primaries is one and the unit of time is fixed in such way that their angular velocity (in the non-rotating frame) is 1. The mass unit is normalized so that $m_1 + m_2=1$, where $m_1$ and $m_2$ are the masses of the primaries. Consequently, the system is fully described by a single parameter, the mass ratio $\mu=m_2/(m_1+m_2)$.

Under this assumptions, there is a constant of motion called {\it Jacobi's Constant}, $C$, given by 
\begin{equation}\label{eq:jacobi}
C(x, y, v_x , v_y):=
-(v_{x}^{2} + v_{y}^{2}) + (x^{2}+y^{2})+ \frac{2(1-\mu)}{r_1}+\frac{2\mu}{r_2}+ \frac{\mu(1-\mu)}{2}.
\end{equation}
Where $r_1=\sqrt{(x+1-\mu)^2+y^2}$ and $r_2=\sqrt{(x-\mu)^2+y^2}$ are the distances to the larger and smaller primary, respectively.

The Jacobi's constant is related to the energy of the system, containing a kinetic term $v_x^2+v_y^2=v^2$ and a potential term, which corresponds to the remaining elements. The kinetic term is always positive. Therefore, for a given value of $C$ fixed, the velocity is given by:
\begin{equation}\label{eq:zeroVelcurve}
F_{C,\mu}(x,y):=v^2 = (x^{2}+y^{2})+ \frac{2(1-\mu)}{r_1}+\frac{2\mu}{r_2}+\frac{\mu(1-\mu)}{2}-C.
\end{equation}

Since $v^2$ cannot be negative, Equation~\ref{eq:zeroVelcurve} imposes a restriction on the configuration space (the $x-y$ plane). Any value $(x,y)$ such that $F_{C,\mu}(x,y) < 0$ cannot be accessed. For a given pair $(C, \mu)$, the massless object cannot enter this forbidden region.  
For fixed values of $C$ and $\mu$, the set $F_{C,\mu}(x,y) = 0$ defines the zero-velocity curve associated with the constant $C$. This curve is of particular interest in celestial mechanics, as it represents the boundary of the feasible region of motion (see \cite{szebehely_theory_1967} for further details).  
Figure~\ref{fig:0VelCurveLevelCurves} shows the zero-velocity curves defined by $ F_{C,0.1}(x,y) = 0$ in the $x-y$ plane for various values of $C$. This example aims to derive a parametric representation $(\xi, y(\xi))$ for a selected curve from this family. As an example we will compute the parametrization for $C=3.5$. This curve is highlighted in red in Figure~\ref{fig:0VelCurveLevelCurves}.
\begin{figure}
    \centering
    \includegraphics[width=0.5\linewidth]{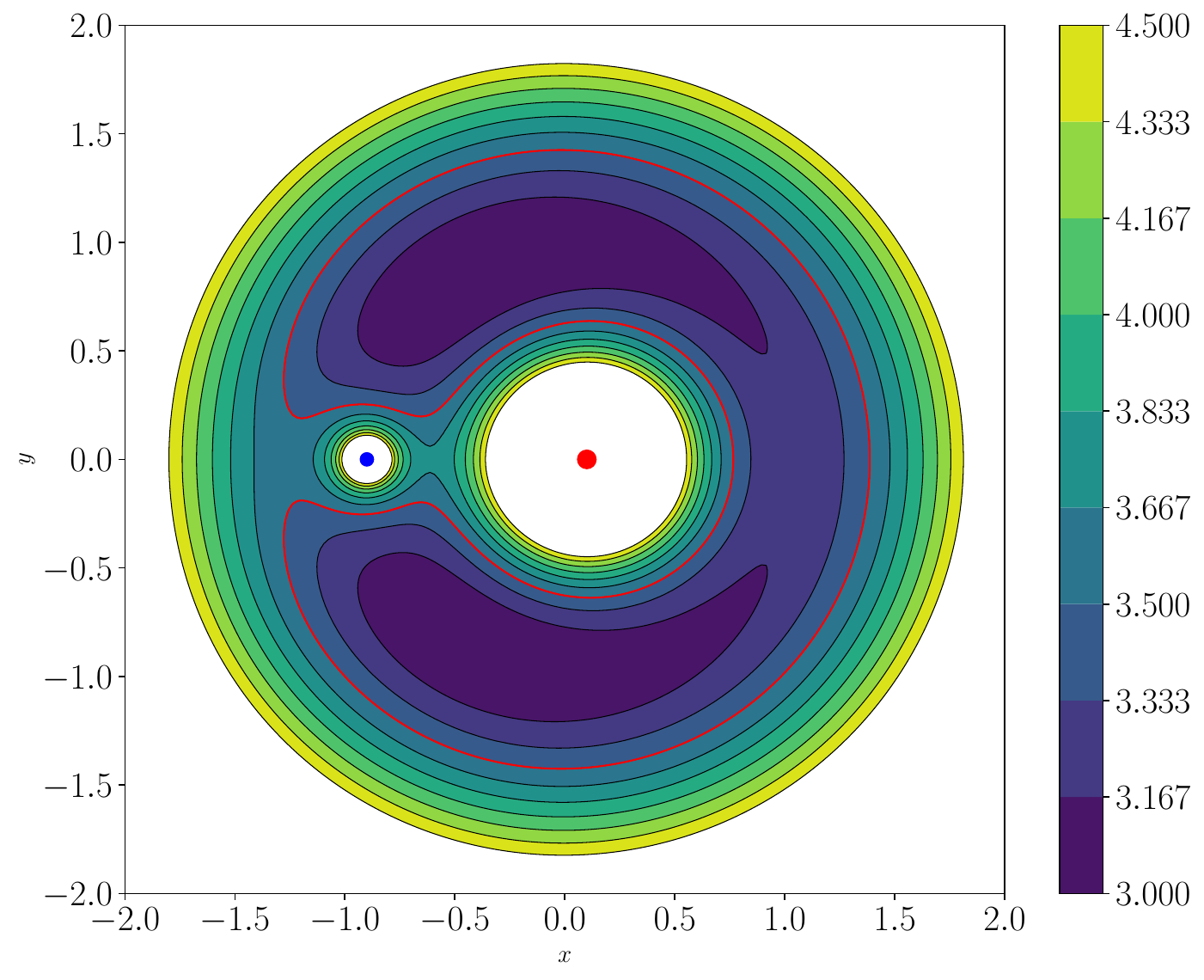}
        \caption{Zero Velocity curves for $F_{C,0.1}(x,y)$ for different values of $C$. The value $C=3.5$ is highlited in red.}
    \label{fig:0VelCurveLevelCurves}
\end{figure}

As in previous cases, we can parametrise the boundary using the Newton-jet procedure. To achieve this, treat $x$ as a parameter in $F_{C,\mu}(x,y)$ and seek a function of the form $y(\xi)$ such that $F_{C,\mu}(x_0 + \xi, y(\xi)) = 0$ in a neighbourhood of $x_0$.
We can repeat the methodology given in the previous examples and compute the iteration $P_i(\xi)$ as:
$$P_{i+1}(\xi)=P_i(\xi)- \dfrac{F_{C,0.1}(x_0+\xi,P_i(\xi))}{F'_{C,0.1}(x_0+\xi,P_i(\xi))}.$$

After applying Newton's method, we obtain the polynomial $P(\xi)$. In accordance with Theorem~\ref{teo:maineps}, we apply Newton's method with jet algebra for a fixed number of iterations beyond what is required for zeroth-order convergence. In this case we use 6 additional iterations to ensure that all terms have converged. Following the methodology of Figure~\ref{fig:academic_evalP}, we perform pointwise comparisons between this solution and a reference. As the exact solution is unavailable, we compute reference values $N(\xi)$ using standard Newton's method for real-valued functions for each value of the parameter $x$. Figure~\ref{fig:zeroVelContInd0} presents the results of this computation. We show the decimal logarithm of the error $|P(\xi) - N(\xi)|$ (purple curve) and the direct evaluations of $F_{C,\mu}(\xi, P(\xi))$ (green curve). 
Both cases exhibits the expected $10^{-16}$ plateau from double-precision arithmetic. It's termination defines the effective radius of convergence as we have seen in the previous examples.%

\begin{figure}
    \centering
    \includegraphics[width=\linewidth]{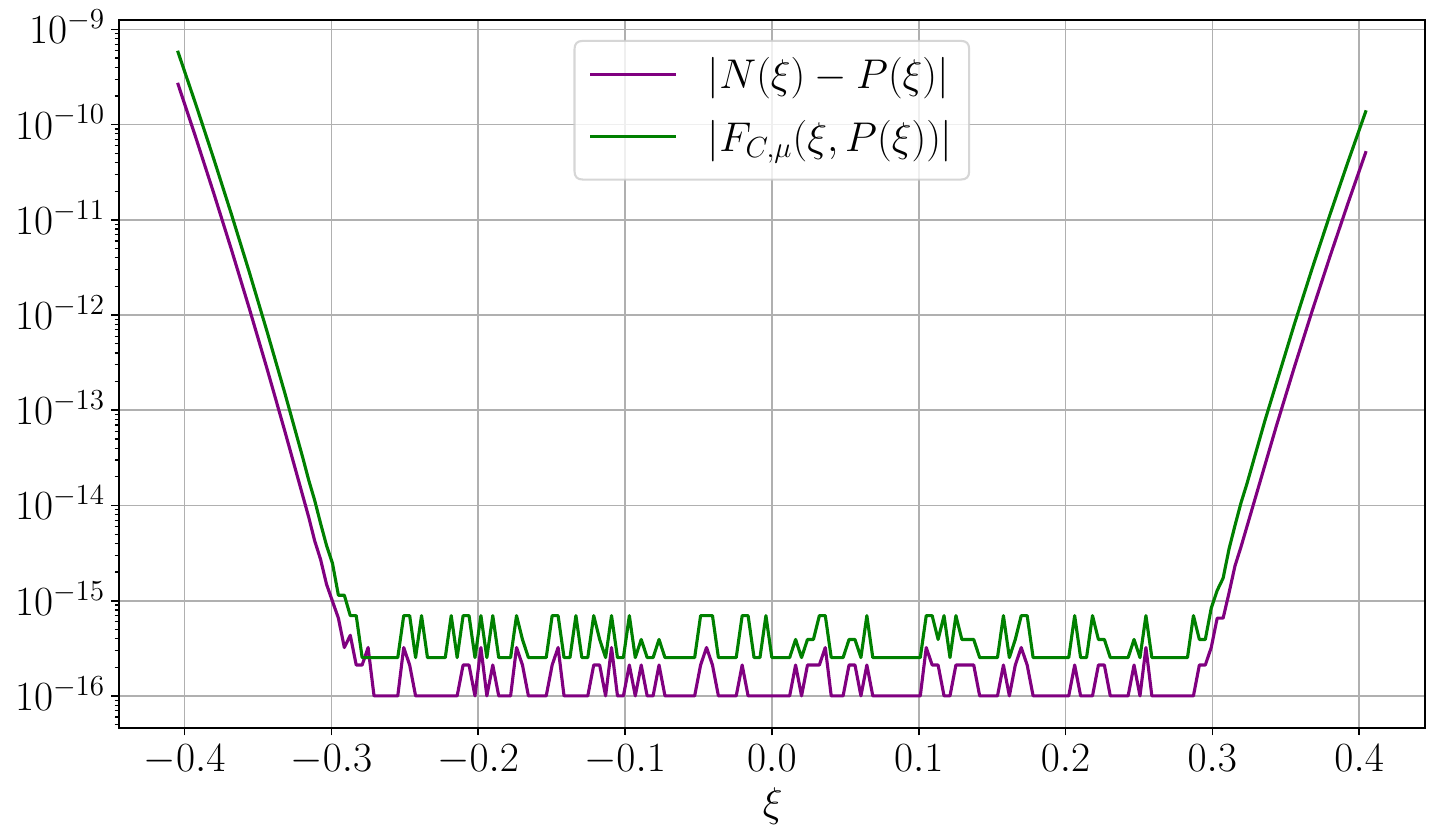}
    \caption{
    Purple: error in logarithmic scale between reference $N(\xi)$ and jet-computed $P(\xi)$; green: direct evaluation $F_{C,\mu}(\xi,P(\xi))$.}
    \label{fig:zeroVelContInd0}
\end{figure}

The easy identification of the plateau given by the effective radius of convergence suggest an algorithm for the continuation of the zero-velocity curve as follows.
Using the Newton-Jet procedure we obtain a polynomial approximation up to order $n$, $P_0(\xi)$, of the solution function $y_0(\xi)$ such that $F_{C,\mu}(\xi,y_0(\xi))=0$. There exists a maximal value of $\xi_m$ for which the approximation is accurate up to a given tolerance, i.e.\ $|P_0(\xi)-y_0(\xi)|<\varepsilon$ for a given tolerance $\varepsilon$ and for all $ \xi<\xi_{0,m}$. Hence, the value $(\xi_{0,m},P(\xi_{0,m}))$ can be used as a new initial seed to find a polynomial approximation up to order $n$, $P_1$, of the function $y_1(\xi)$ that solves $F_{C,\mu}(\xi_{0,m}+\xi,y_1(\xi))=0$. This can be iterated to consecutively obtain new continuation points.
Given a desired tolerance $\varepsilon$, the algorithm runs as follows:
\begin{enumerate}
    \item Set an initial value $x_0=0$ and compute $P_0(\xi)$ as the polynomial approximation up to order $n$ of $y_0(\xi)$ such that $F_{C,\mu}(\xi,y(\xi))=0$ using Newton's Jet procedure.
    \item Compute $\xi_{i,m} = \max \{ \tilde {\xi} \;\mbox{ s.t. }\; |P_i(\xi)-y_i(\xi)|<\varepsilon \;\; \forall \xi<\tilde{\xi}\}$. The following bisection method can be used:
    \begin{enumerate}
        \item Fix $\xi_-=0$ and $\xi_+=1$. While $F_{C,\mu}(\xi_+,y(\xi_+))<\varepsilon$, duplicate the value of $\xi_+$, i.e. $\xi_+=2\xi_+$.
        \item Compute $\xi_h$ as the midpoint between $\xi_-$ and $\xi_+$. If $F_{C,\mu}(\xi_h,y(\xi_h))<\varepsilon$ then $\xi_-=\xi_h$, otherwise $\xi_+=\xi_h$.
        \item Repeat steps (a) and (b) until $\xi_-$ and $\xi_+$ are close enough or $F_{C,\mu}(\xi_h,y(\xi_h))$ is close enough to $\varepsilon$
    \end{enumerate}
    \item Compute the initial seed for the next step as $p_{i+1}=P_i(\xi_{i,m})$ and the value of the parameter $x_{i+1}=\xi_{i}+\xi_{i,m}$.
    \item Use Newton's Jet to obtain the polynomial approximation up to order $n$ of $y_{i+1}(\xi)$, $P_{i+1}(\xi)$, such that $F_{C,\mu}(x_{i+1}+\xi,y_{i+1}(\xi))=0$.
    \item Repeat steps 2 to 4 until the conditions of ending the continuation are met.
\end{enumerate}

Using this algorithm, it is possible to prolong the zero-velocity curve further than the initial effective radius of convergence. However, the radius of convergence decreases whenever the curve approaches a point $\xi_\infty$ where $\displaystyle y'(\xi)\mathop{\longrightarrow}\limits_{\xi\to \infty}\pm\infty$. In that case, there does not exist any polynomial approximation to the curve at $\xi_\infty$.

As an alternative, the previous algorithm has been modified so that the Newton-jet method computes polynomial approximations of $y(x)$ or $x(y)$, depending on the slope:
\begin{itemize}
    \item If $|y'(x)| < 1$, the continuation is performed by treating $x$ as a parameter in the Newton-jet procedure, yielding a polynomial approximation of $y(x)$.
    \item If $|y'(x)| > 1$, $y$ is treated as the parameter, and a polynomial approximation of $x(y)$ is obtained.
\end{itemize}
Since at a given pair $(x_0,y_0)$ such that $y_0=y(x_0)$ and $x_0=x(y_0)$ it is verified that $y'(x_0)=1/x'(y_0)$, the algorithm changes the coordinate that serves as a parameter each time that the derivative of the polynomial evaluated at $\xi_m$ is bigger than 1. In the limit case $|y'(x)|=|x'(y)|=1$, the next step is done without changing the previous configuration.

The algorithm was iterated until a complete circuit of the zero-velocity curve was obtained, which took 48 iterations. This procedure successfully extends the curve (Figure~\ref{fig:0VelCurve}, black line). The coloured triangles superimposed on the curve mark the application points of Newton's method. Their colour is a function of the local effective radius of convergence \(\xi_{i,m}\), and their orientation denotes the local continuation parameters: a horizontal orientation (pointing left/right) signifies the use of the \(y(x)\) parametrization, while a vertical orientation (pointing up/down) signifies the \(x(y)\) parametrization. The pointing direction indicates the local continuation direction along the curve.

\begin{figure}
    \centering
    \includegraphics[width=0.49\linewidth]{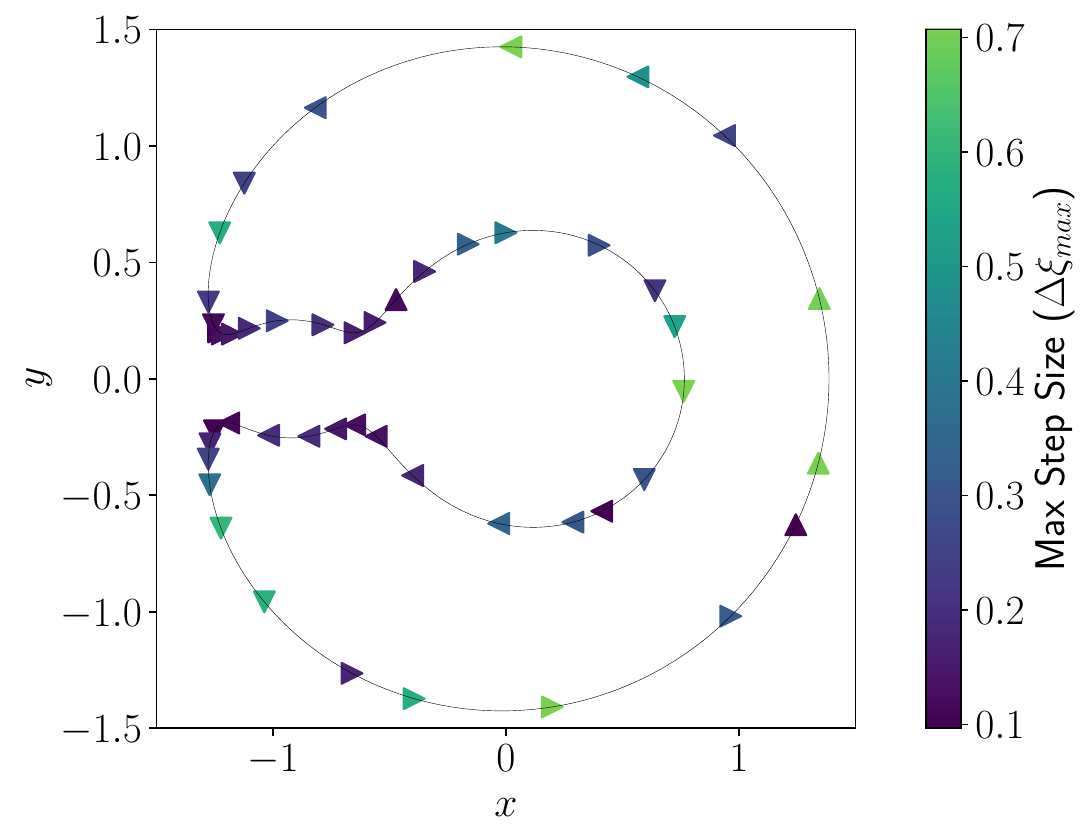}
    \includegraphics[width=0.405\linewidth]{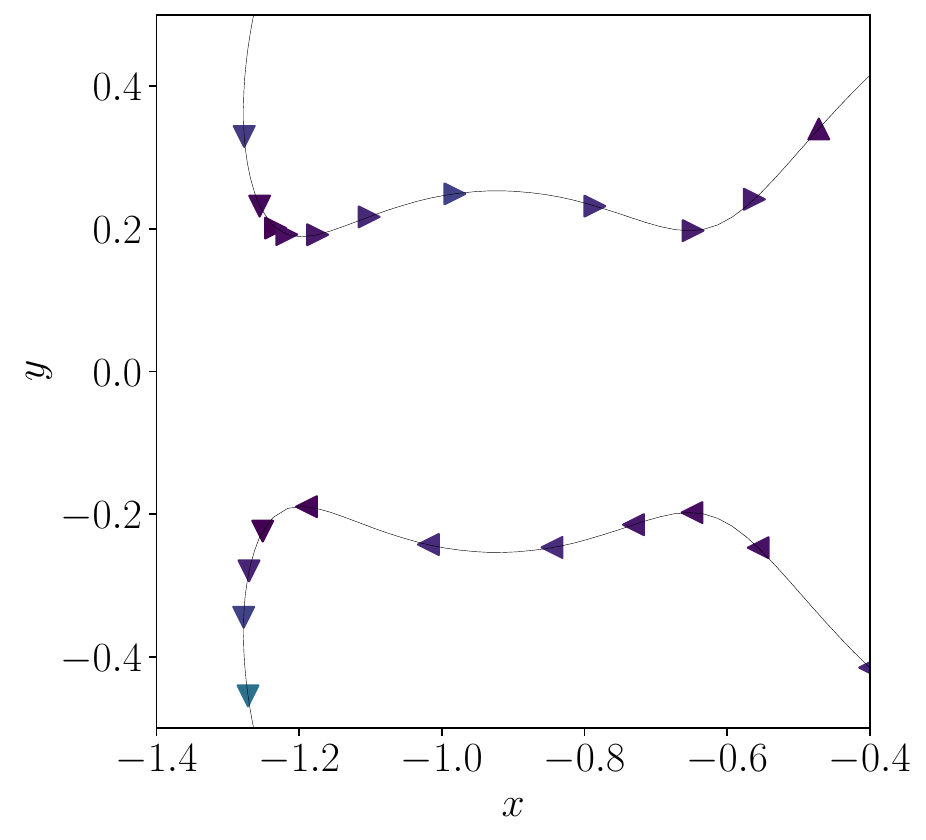}
    \caption{\textbf{Left panel:} Reconstruction of the zero-velocity curve \( F_{3.5,0.1}(x,y) = 0 \) (solid black line) via the algorithm described in Section~\ref{sec:zerovelocity} Triangles denote the central points for Newton iterations. Their orientation specifies the local parametrization (vertical: \( x(y) \); horizontal: \( y(x) \)) and the continuation direction (left, right, up or down). \textbf{Right panel:} Close-up view of the left portion of the curve.
            }
    \label{fig:0VelCurve}
\end{figure}

As a final observation,  note that the recovered curve is composed of several parametrisation segments. This implies that most points admit two valid parametrisations: one from the preceding segment and another from the succeeding segment. To illustrate this, each continuation segment is rendered in a different colour in Figure~\ref{fig:0velColours}. Each segment is parametrised by an arc-length parameter over the interval $[\xi_{i,\min},\xi_{i,\max}] = [\xi_i - \xi_{i,m},\, \xi_i + \xi_{i,m}]$. This symmetric interval about the central point $\xi_i$ is chosen under the heuristic that the polynomial approximation remains effective within this range observed in previous examples. In the figure, the Newton iteration point for a segment is located at the cross marking the end of the preceding segment. For a given segment, a circle marks its start $(\xi_{i,\min})$ and a cross marks its end $(\xi_{i,\max})$. Since each parametrisation interval spans both positive and negative values in the parameter space, the segments naturally overlap. In extreme cases, up to three parametrisations coincide at a single point. These cases, such as the olive segment near $(-0.75,-1.2)$, correspond precisely to points where the algorithm switches between horizontal ($y(x)$) and vertical ($x(y)$) parametrisations.

\begin{figure}
    \centering
    \includegraphics[width=0.8\linewidth]{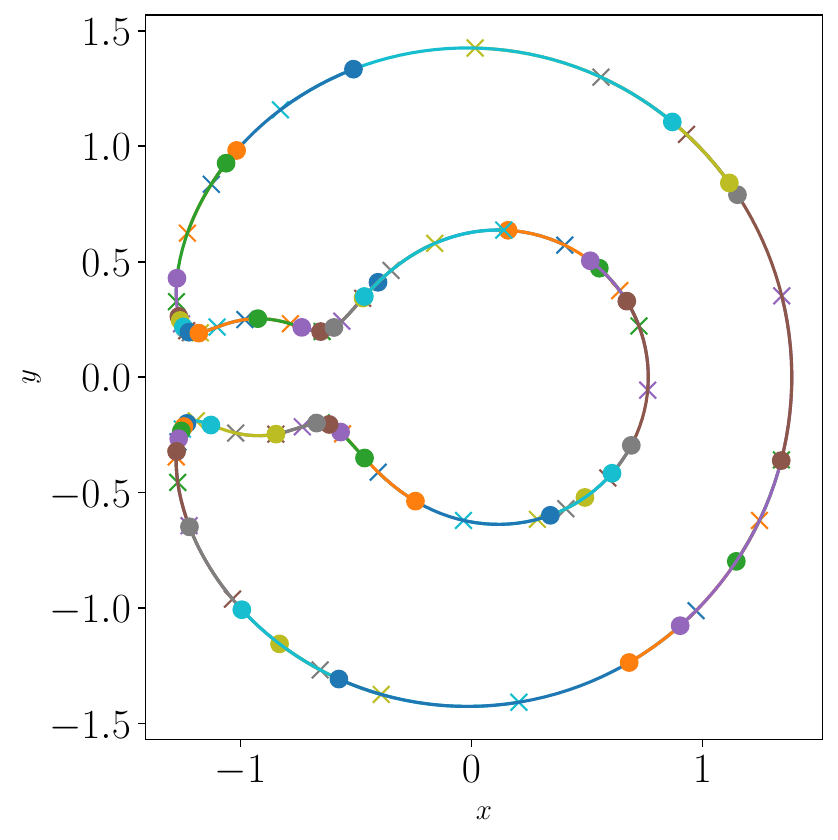}
    \caption{Parametrisation of a zero-velocity curve for \( F_{3.5,0.1}(x,y) = 0 \). Colours indicate distinct continuation segments (start: circle, end: cross).}
    \label{fig:0velColours}
\end{figure}

\section{Conclusions}\label{sec:conclusions}

This paper has developed and analysed a method for solving parameter-dependent non-linear equations using a polynomial algebra. The core idea is to approximate the implicit solution $x(c_0+\xi)$ of $f(x, c_0+\xi) = 0$ as a jet $P(\xi)$ and to employ polynomial arithmetic within a Newton-type iteration.

The central theoretical contribution is Theorem~\ref{teo:main}, which proves that each iteration of the Jet-Newton method doubles the number of correct coefficients in the Taylor expansion of the solution. This technique transforms Newton's method from a local root-finder into a high-order series generator, effectively providing a semi-global polynomial approximation of the solution branch.

To address practical considerations, Theorem~\ref{teo:maineps} characterizes the effect of small errors in the initial seed. It shows that the error propagates quadratically within a given polynomial order while simultaneously doubling the number of accurate coefficients at that error level, thereby improving the overall approximation.

The practical power of the approach is exemplified through three distinct examples:
\begin{itemize}
    \item An academic example provides insight into the method and highlights practical considerations when working in double and quadruple precision.
    \item The solution of Kepler's equation showcases specific cases where the absence of subtractive cancellation allows the method to perform slightly better than the theoretical bounds predict.
    \item A novel continuation algorithm is developed to compute zero-velocity curves in the circular restricted three-body problem, illustrating the method's robustness for geometric continuation.
\end{itemize}

The method is proven for simple zeros of sufficiently smooth, one-dimensional functions with a single parameter and has been shown to be effective for the problems studied. However, its application to higher-dimensional systems and degenerate singularities remains a subject for future investigation. Natural extensions include applying the framework to bifurcation detection and enhancing other high-order iterative schemes.

\bibliographystyle{plain} 
\bibliography{biblio} 


		


\end{document}